\documentclass[a4paper,reqno]{amsart}

\usepackage{amsmath, amsthm, amssymb, mathrsfs}
\usepackage{xcolor}
\usepackage{tikz}
\usepackage{amsrefs}
\usepackage{microtype}
\usepackage{pdfpages}
\usepackage[normalem]{ulem}

\usepackage{pgfplots}
\pgfplotsset{compat=1.15}
\usepackage{mathrsfs}
\usetikzlibrary{arrows}

\newcommand{\bZ}{\mathbb{Z}}
\newcommand{\bR}{\mathbb{R}}

\newtheorem{theorem}{Theorem}
\newtheorem{lemma}{Lemma}
\newtheorem{claim}{Claim}
\newtheorem*{isoperimetriclemma}{Lemma~\ref{isoperimetric}}

\title{A tight bound for the number of edges of matchstick graphs}
\author{J\'er\'emy Lavoll\'ee and Konrad Swanepoel}

\subjclass[2020]{Primary 52C10. Secondary 05C10}
\keywords{matchstick graph, penny graph, plane unit-distance graph}

\begin{document}

\maketitle
\begin{abstract}
    A matchstick graph is a plane graph with edges drawn as unit-distance line segments.
    Harborth introduced these graphs in 1981 and conjectured that the maximum number of edges for a matchstick graph on $n$ vertices is $\lfloor 3n-\sqrt{12n-3} \rfloor$.
    In this paper we prove this conjecture for all $n\geq 1$.
    The main geometric ingredient of the proof is an isoperimetric inequality related to L'Huilier's inequality.
\end{abstract}

\section{Introduction}
A \emph{matchstick graph} is a graph drawn in the Euclidean plane with each edge a straight unit-length segment, such that two edges only intersect in a common endpoint. 
These graphs were introduced by Harborth \cites{oberwolfach, lighter} in 1981.
Most of the literature deals with matchstick graphs which are regular \cites{blokhuis, gerbracht0, gerbracht, kurz-4-regular, kurz, KM, winkler, WDV3} or almost regular \cites{LS2, winkler2, WDV0, WDV, WDV2}.
There are also a few papers dealing with more general aspects, such as enumeration \cites{salvia} and algorithmic recognition \cites{abel, kurz-recognition}.
Harborth considered the extremal problem of finding the minimum number $n$ of vertices in a matchstick graph on $e$ edges, or equivalently, the maximum number of edges for a given number of vertices.
He conjectured in \cite{lighter} that $e\leq 3n-\sqrt{12n-3}$ (see also \cite{research}*{p.~225}) and proved this in the special case of penny graphs using a neat induction on the number of vertices \cites{harborth} (see also \cite{PachAgarwal}*{Theorem~13.12}).
A \emph{penny graph} is a matchstick graph with the additional property that the circles of radius $1/2$ around the vertices are non-overlapping.

\begin{theorem}[Harborth \cite{harborth}]\label{thm0}
Let $G$ be a matchstick graph on $n$ vertices such that the distance between any two vertices is at least $1$.
Then the number of edges of $G$ satisfies $e\leq 3n-\sqrt{12n-3}$.
\end{theorem}
For each $n\geq 1$ there are examples on the triangular lattice with $n$ vertices and the optimal $\lfloor 3n-\sqrt{12n-3}\rfloor$ edges \cites{harborth, PachAgarwal}.
Proving Harborth's conjecture in general turns out to be trickier.
Using the isoperimetric inequality (Lemma~\ref{isoperimetric0}) we showed in \cite{lavollee} that $e\leq 3n-c\sqrt{12n-3}$, where $c=\frac12(1+\sqrt{\pi\sqrt{3}/6})=0.976\dots$, which implies Harborth's conjecture for some small values of $n$.
In this paper we settle the conjecture completely.
\begin{theorem}\label{thm1}
Let $G$ be a matchstick graph with $n$ vertices and $e$ edges.
Then $e \leq 3n-\sqrt{12n-3}$.
\end{theorem}
As in our previous paper \cite{lavollee}, as a first step we use the Euler formula and the isoperimetric inequality.
\begin{lemma}[Isoperimetric inequality]\label{isoperimetric0}
For any simple polygon of perimeter $b$ and area $A$ we have $4\pi A < b^2$.
\end{lemma}
However, we will then exploit a different isoperimetric inequality that takes into account that many edges on the outer boundary of $G$ lie on the same triangular lattice.
A variant of the isoperimetric inequality in the plane, known as L'Huilier's inequality, states that among all polygons of a given perimeter for which the sides are constrained to be parallel to a given set of directions, the one of maximum area is circumscribed to a circle \cite{FT}*{pp.~9--10}, \cite{FTFTK}*{pp.~12--13}.
In particular, among all closed polygons of a given perimeter with each side parallel to one of the sides of a fixed regular hexagon, the regular hexagon is optimal.
We prove the following variant of this special case where we allow a certain bounded amount of the perimeter to be unconstrained in direction.

\begin{lemma}\label{isoperimetric}
Let $P$ be a simple polygon with perimeter $b$, area $A$, and with the total length of the sides not parallel to any side of some fixed regular hexagon at most $b_*$.
Then \[8\sqrt{3}A \leq \bigl(b+ \bigl(\tfrac{2}{\sqrt{3}}-1\bigr)b_*\bigr)^2.\]
\end{lemma}

Our proof of Theorem~\ref{thm1} has an analytical flavour, with many inequalities appearing in it that are rather weak, with constants that could easily be improved.
Surprisingly, this turns out to be unnecessary.
The inequalities also often need $n$ to be large, which creates the potential problem that this approach will only work for sufficiently large $n$.
However, the original isoperimetric inequality allows us to dismiss all values of $n < 147$ early on in the proof.
Then the rest of the proof goes through, although there are many inequalities that have to be checked.
This can be done by hand or by apps such as Wolfram Alpha or Desmos.
We have to balance the asymptotics with constant terms in order to make the inequalities valid for small $n$.
It is rather strange that such an approach manages to make the induction succeed for all values of $n$.

\subsection{Proof outline}
The proof of Theorem~\ref{thm1} is in Section~\ref{sec:proof}.
Here we give a high-level summary.
The proof is by induction with the induction step split up into $12$~claims.
As the cases $1\leq n\leq 4$ of Theorem~\ref{thm1} are easily checked, we assume that $n > 4$ and that the theorem holds for all smaller values of $n$, but not for $n$, and aim for a contradiction.

Starting off, we show that the graph has minimum degree $3$ (Claim~\ref{mindegree}) and is $2$-connected (Claim~\ref{2-connected}).
We then apply the Euler formula and count incident edge-face pairs to give an upper bound for the length of the boundary of $G$ in terms of $n$ and a certain weighted count $F$ of the number of non-triangular inner faces (Claim~\ref{b-bound}), as well as a lower bound for the number of triangles (Claim~\ref{f3bound}).
If this weighted count $F$ equals $0$ or $1$, then either all inner faces are triangles, a case where Theorem~\ref{thm0} applies, or there is a quadrilateral as well, a case that can be easily dealt with by splitting up the graph at the quadrilateral and applying induction.
Thus we can assume that $F\geq 2$ (Claim~\ref{Fgeq2}).

We apply the isoperimetric inequality (Lemma~\ref{isoperimetric0}) to our upper bound of the boundary length of $G$ and the lower bound on its area derived from the lower bound on the number of triangles.
This gives an upper bound for $F$ of the form $c\sqrt{12n-3}$ (we use $c=1/11$, but $c$ can be made as small as $1/20$) as well as showing that we may assume that $n\geq 147$ otherwise $F < 2$ (Claim~\ref{Fbound1}).

We next consider the so-called lattice components of the graph, which we define to be maximal $2$-connected subgraphs that each lie on some triangular lattice.
We find lower bounds for their boundary lengths (Claim~\ref{lattice-boundary}) and show that they cover almost all of $G$ without too much overlap (Claim~\ref{lattice-components}).
This enables us to show that the largest lattice component, $G_1$, has to be quite large (we show that $n(G_1) > 3n/4$ in Claim~\ref{n1bound}, but it is possible to go up to $0.97n$ if $n$ is sufficiently large).

We split $G$ up into $G_1$ and a slight enlargement of $G-G_1$, and apply induction to obtain a lower bound for $F$ in terms of $n$ and $n_1$ (Claim~\ref{eq1}) that will be crucial for the final part of the proof.
We also bound the number of edges on the boundary of $G$ that are not on the boundary of $G_1$ (Claim~\ref{b*upperbound}).
Thus, except for a bounded quantity, the boundary edges of $G$ all lie on the same triangular lattice.
Our new isoperimetric inequality (Lemma~\ref{isoperimetric}) then gives an improved upper bound on the area of $G$, which can be bounded from below using the bound on the number of triangles from Claim~\ref{f3bound}.
The upshot of this is that we find a very good upper bound on $F$ (Claim~\ref{Fbound2}), which, together with the lower bound from Claim~\ref{eq1} gives us our final contradiction.

\subsection{Definitions}
We define a \emph{triangular lattice} to be any subset of the plane isometric to $\{m(1,0)+n(1/2,\sqrt{3}/2) : m,n\in\bZ\}$, and we say that a finite set of points \emph{lies on a triangular lattice} if it is a subset of some triangular lattice.
We will use the following simple observation.
\begin{lemma}\label{latticelemma}
If $a$ and $b$ are distinct points on a triangular lattice and $c$ is a point at distance $1$ to both $a$ and $b$, then $c$ lies on the same triangular lattice.
\end{lemma}
We use standard graph theory terminology as can be found in \cite{Diestel}*{Sections 1.1--1.4, 4.1, 4.2}.
We call the cycle bounding the outer face of a plane graph $G$ the \emph{boundary} of $G$, its length the \emph{boundary length} of $G$, and its edges the \emph{boundary edges} of $G$.

We will occasionally use the function $\phi(x)=\sqrt{12x-3}-3$ which is non-negative for all $x\geq 1$.
Since $\phi$ is strictly concave, $\phi(x+a)-\phi(x)$ is strictly decreasing in $x$ for all fixed $a>0$.
This implies the following inequality that we will use repeatedly.
\begin{lemma}\label{concave}
$\phi(a)+\phi(b) < \phi(a-c) + \phi(b+c)$ whenever $a> b+c$, $b\geq 1$ and $c > 0$.
\end{lemma}

\section{Proof}\label{sec:proof}
\subsection{Setting up the induction}
Let $G$ be a matchstick graph on $n=n(G)$ vertices in the plane.
Denote the number of edges by $e=e(G)$.
We prove Theorem~\ref{thm1} 
by induction.
The cases $1\leq n\leq 3$ are trivial. The case $n=4$ follows since the complete graph on $4$ vertices is not a matchstick graph.
So we fix $n > 4$ and assume the theorem is true for all matchstick graphs of less than $n$ vertices.
Among all matchstick graphs on $n$ vertices, fix one $G$ with the maximum number $e$ of edges.
We assume that
\begin{equation}\label{assumption1}
e > 3n-\sqrt{12n-3},
\end{equation}
and will aim for a contradiction.
Claims~\ref{mindegree}--\ref{Fbound2} below will all have as unstated assumptions the induction hypothesis, $n > 4$, the inequality \eqref{assumption1}, and that $G$ has the maximum number of edges among all matchstick graphs with $n$ vertices.
\subsection{Basic properties}
\begin{claim}\label{mindegree}
Each vertex of $G$ has at least $3$ neighbours.
\end{claim}
\begin{proof}
Suppose that there is a vertex $v$ with at most $2$ neighbours, and let $G'=G-v$.
Then by induction and Lemma~\ref{concave},
\begin{align*}
    e&\leq e(G')+2\leq 3(n-1)-\sqrt{12(n-1)-3}+2\\
    &= 3n-\phi(n-1)-4\\
    &\leq 3n-\phi(n)+\phi(5)-\phi(4)-4\\
    &= 3n-\sqrt{12n-3} + \phi(5)-\phi(4)-1 < 3n-\sqrt{12n-3},
\end{align*}
which contradicts \eqref{assumption1}.
\end{proof}

\begin{claim}\label{2-connected}
$G$ is 2-connected.
\end{claim}
\begin{proof}
$G$ has to be connected, otherwise we can move a connected component until one of its vertices has distance $1$ to some other vertex, contradicting maximality.

If $G$ is not $2$-connected, then $G$ has a cut vertex.
Thus there exist two subgraphs $G_1$ and $G_2$ that cover $G$ and have only a single vertex in common.
Then with $n_1=n(G_1), n_2=n(G_2)\geq 2$, $e_1=e(G_1)$, and $e_2=e(G_2)$, we have $n = n_1+n_2-1$, and 
\begin{align*}
e &= e_1+e_2\\
&\leq 3n_1-\sqrt{12n_1-3}+3n_2-\sqrt{12n_2-3}\\
&= 3n-\phi(n_1)-\phi(n_2)-3\\
&< 3n-\phi(n_1+n_2-1)-\phi(1)-3 = 3n-\sqrt{12n-3}
\end{align*}
by induction and Lemma~\ref{concave}, which contradicts \eqref{assumption1}.
\end{proof}

\subsection{The Euler formula and double counting.}
Since $G$ is 2-connected by Claim~\ref{2-connected}, all of its faces are polygons.
Let $b$ denote the length of the outer face, and for each $i\geq 3$, let $f_i$ denote the number of inner faces bounded by a cycle of length $i$.
Since $G$ is connected, Euler's formula gives
\begin{equation}\label{Euler}
n-e+\sum_{i\geq 3}f_i =1,
\end{equation}
and since $G$ is $2$-connected, by counting incident edge-face pairs in two ways, we obtain
\begin{equation}\label{double-counting}
2e=b+\sum_{i\geq 3}if_i.
\end{equation}
Let $F=\sum_{i\geq 4}(i-3)f_i$.
Then \eqref{Euler} and \eqref{double-counting} imply
\begin{equation}\label{basic-identity}
e=3n-3-b-F.
\end{equation}
From \eqref{assumption1} we then obtain
\begin{claim}\label{b-bound}
$b < \sqrt{12n-3} - 3 - F =\phi(n)-F$.
\end{claim}

\subsection{Using the isoperimetric inequality}
We need the following lower bound on the number of triangular faces, so that we can estimate the area of the region bounded by $G$, then use the isoperimetric inequality to find a lower bound for $b$ and from that our first upper bound on $F$.
This upper bound will be needed later in our proof of an improved upper bound for $F$ (Claim~\ref{Fbound2}).
\begin{claim}\label{f3bound}
$f_3 > 2n+1-\sqrt{12n-3}-F=\phi(n)^2/6-F$.
\end{claim}
\begin{proof}
Using the Euler formula~\eqref{Euler} and our assumption~\eqref{assumption1}, we obtain
\begin{equation*}
    f_3 = e-n+1-\sum_{i\geq 4} f_i 
    > 2n-\sqrt{12n-3}+1-F.\qedhere
\end{equation*}

\end{proof}

\begin{claim}\label{Fgeq2}
$F\geq 2$.
\end{claim}
\begin{proof}
If $F=0$, then all inner faces are triangular, and since $G$ is $2$-connected, it lies on a single triangular lattice.
Then $G$ is a penny graph, and Harborth's Theorem~\ref{thm0} gives $e\leq 3n-\sqrt{12n-3}$, a contradiction.

If $F=1$, then there is a single quadrilateral inner face with all other inner faces triangular.
Each vertex $v$ of this quadrilateral must lie on the boundary of $G$, otherwise, all other faces incident with $v$ would be equilateral triangles, implying that the angle of the quadrilateral at $v$ is a multiple of $60^\circ$.
Then the quadrilateral would have angles $60^\circ$ and $120^\circ$, and we could add an edge between the two opposite vertices at the $120^\circ$ angles, violating the maximality of $G$.

We next decompose $G$.
Since $n>5$, at least one of the edges of the quadrilateral is not a boundary edge of $G$.
Let $p$ and $q$ be the endpoints of such an edge.
Then $G - \{p,q\}$ is not connected and has exactly two connected components $C_1$ and $C_2$.
Let $G_i$ be the subgraph of $G$ induced by the vertices of $C_i$ and $p$ and $q$ for $i=1,2$.
Then $G_1$ and $G_2$ have only an edge of the quadrilateral in common, and cover $G$.
Denoting $n_i=n(G_i)$ and $e_i=e(G_i)$, $i=1,2$, we have $n_1+n_2=n+2$ and $e_1+e_2=e+1$.
Since $pq$ is not on the boundary of $G$, it follows that $n_i\geq 3$ ($i=1,2$), and the length $b$ of the boundary of $G$ is at least $5$.
Then Claim~\ref{b-bound} implies that $n >7$, and
\begin{align*}
    e &= e_1+e_2-1\\
    &\leq 3n-\phi(n_1)-\phi(n_2)-1 \quad\text{by induction}\\
    &\leq 3n-\phi(n_1+n_2-3)-\phi(3)-1 \quad\text{by Lemma~\ref{concave}}\\
    &< 3n-\phi(n)+\phi(7)-\phi(6)-\phi(3)-1 \quad\text{again by Lemma~\ref{concave}}\\
    &< 3n-\sqrt{12n-3},
\end{align*}
contradicting assumption~\ref{assumption1}.
\begin{figure}
    \centering
\definecolor{ffqqqq}{rgb}{1,0,0}
\definecolor{qqqqct}{rgb}{0,0,0.7647058823529411}
\definecolor{bvqqqq}{rgb}{0.7098039215686275,0,0}
\definecolor{qqqqff}{rgb}{0,0,1}
\definecolor{uuuuuu}{rgb}{0.26666666666666666,0.26666666666666666,0.26666666666666666}
\definecolor{qqwuqq}{rgb}{0,0.39215686274509803,0}
\begin{tikzpicture}[line cap=round,line join=round,>=triangle 45,scale=1.2]
\draw[line width=1pt, color=qqqqff] (-0.115834378161942,0.90319744372555) -- (0.8546048704318957,0.6618519138547284) -- (0.5783966060929246,1.6229497209018848) -- cycle;
\draw[line width=1pt, color=qqqqff] (0.5783966060929246,1.6229497209018848) -- (0.8546048704318957,0.6618519138547284) -- (1.5488358546867627,1.3816041910310628) -- cycle;
\draw[line width=1pt, color=qqqqff] (1.5488358546867627,1.3816041910310628) -- (0.8546048704318957,0.6618519138547284) -- (1.8250441190257332,0.4205063839839058) -- cycle;
\draw[line width=1pt, color=qqqqff] (0.8546048704318957,0.6618519138547284) -- (0.872577031426156,-0.3379865738167537) -- (1.7294764809339438,0.17749701800092027) -- cycle;
\draw[line width=1pt, color=qqqqff] (1.7294764809339438,0.17749701800092027) -- (0.872577031426156,-0.3379865738167537) -- (1.7474486419282036,-0.822341469670562) -- cycle;
\draw[line width=1pt, color=qqqqff] (1.7294764809339438,0.17749701800092027) -- (1.7474486419282036,-0.822341469670562) -- (2.6043480914359916,-0.3068578778528882) -- cycle;
\draw[line width=1pt, color=qqqqff] (-0.09786221716768187,-0.09664104394593212) -- (-0.115834378161942,0.90319744372555) -- (-0.9727338276697298,0.38771385190787616) -- cycle;
\draw[line width=1pt, color=ffqqqq] (-0.9727338276697298,0.38771385190787616) -- (-0.115834378161942,0.90319744372555) -- (-0.9907059886639896,1.3875523395793583) -- cycle;
\draw[line width=1pt, color=ffqqqq] (-0.09786221716768187,-0.09664104394593212) -- (-0.9727338276697298,0.38771385190787616) -- (-0.9547616666754701,-0.6121246357636062) -- cycle;
\draw[line width=1pt, color=ffqqqq] (-0.9727338276697298,0.38771385190787616) -- (-0.9907059886639896,1.3875523395793583) -- (-1.8476054381717777,0.8720687477616847) -- cycle;
\draw[line width=1pt, color=ffqqqq] (-0.9727338276697298,0.38771385190787616) -- (-1.8476054381717777,0.8720687477616847) -- (-1.8296332771775186,-0.12776973990979779) -- cycle;
\draw[line width=1pt, color=ffqqqq] (-0.9547616666754701,-0.6121246357636062) -- (-0.9727338276697298,0.38771385190787616) -- (-1.829633277177518,-0.12776973990979731) -- cycle;
\draw[line width=1pt, color=qqqqff] (0.872577031426156,-0.3379865738167537) -- (-0.09786221716768187,-0.09664104394593212) -- (0.17834604717128927,-1.0577388509930885) -- cycle;
\draw [line width=1pt, color=ffqqqq] (-0.123,-0.09664104394593212)-- (-0.143,0.90319744372555);
\draw [line width=1pt, color=qqqqff] (-0.078,-0.09664104394593212)-- (-0.098,0.90319744372555);
\draw [line width=1pt, color=qqqqff] (-0.09786221716768187,-0.09664104394593212)-- (0.872577031426156,-0.3379865738167537);
\draw [line width=1pt, color=qqqqff] (0.8546048704318957,0.6618519138547284)-- (0.872577031426156,-0.3379865738167537);
\draw [line width=1pt, color=qqqqff] (-0.115834378161942,0.90319744372555)-- (0.8546048704318957,0.6618519138547284);
\draw [fill=black] (-0.09786221716768187,-0.09664104394593212) circle (1pt);
\draw[color=black] (-0.17,-0.33) node {$q$};
\draw [fill=black] (-0.115834378161942,0.90319744372555) circle (1pt);
\draw[color=black] (-0.15,1.1) node {$p$};
\draw [fill=black] (0.872577031426156,-0.3379865738167537) circle (1pt);
\draw [fill=black] (0.8546048704318957,0.6618519138547284) circle (1pt);
\draw[color=ffqqqq] (-0.557366590219986,1.4) node {$G_1$};
\draw[color=qqqqff] (0.2,1.6) node {$G_2$};
\draw [fill=black] (0.5783966060929246,1.6229497209018848) circle (1pt);
\draw [fill=black] (1.5488358546867627,1.3816041910310628) circle (1pt);
\draw [fill=black] (1.8250441190257332,0.4205063839839058) circle (1pt);
\draw [fill=black] (1.7294764809339438,0.17749701800092027) circle (1pt);
\draw [fill=black] (1.7474486419282036,-0.822341469670562) circle (1pt);
\draw [fill=black] (2.6043480914359916,-0.3068578778528882) circle (1pt);
\draw [fill=black] (-0.9727338276697298,0.38771385190787616) circle (1pt);
\draw [fill=black] (-0.9907059886639896,1.3875523395793583) circle (1pt);
\draw [fill=black] (-0.9547616666754701,-0.6121246357636062) circle (1pt);
\draw [fill=black] (-1.8476054381717777,0.8720687477616847) circle (1pt);
\draw [fill=black] (-1.8296332771775186,-0.12776973990979779) circle (1pt);
\draw [fill=black] (-1.829633277177518,-0.12776973990979731) circle (1pt);
\draw [fill=black] (0.17834604717128927,-1.0577388509930885) circle (1pt);
\end{tikzpicture}
    \caption{Subgraphs $G_1$ and $G_2$ covering $G$ and sharing edge $pq$ of a quadrilateral}\label{Fgeq2fig}
\end{figure}
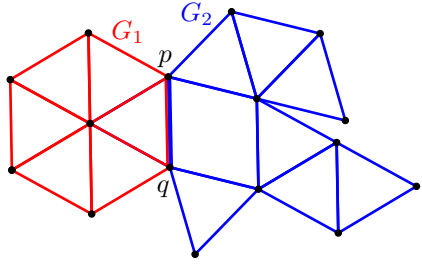
\end{proof}
\begin{claim}\label{Fbound1}
$n \geq 147$ and $F < \frac{1}{11}\sqrt{12n-3}-1$.
\end{claim}
\begin{proof}
Denote the area of the boundary polygon of $G$ by $A$.
The area of an equilateral triangle of side length $1$ is $\sqrt{3}/4$, hence $A > \frac{\sqrt{3}}{4}f_3$.
By Claim~\ref{f3bound}, $f_3 > \phi(n)^2/6-F$, and by the isoperimetric inequality (Lemma~\ref{isoperimetric0}) and Claim~\ref{b-bound}, $4\pi A < b^2 < (\phi(n)-F)^2$.
In the remainder of this proof we write $\phi=\phi(n)$.
We thus obtain the inequality $\frac{\pi\sqrt{3}}{6}(\phi^2-6F) < (\phi-F)^2$, or when expanded,
\begin{equation}\label{eq:quadratic}
    F^2-(2\phi-\pi\sqrt{3})F+(1-\pi\sqrt{3}/6)\phi^2>0.
\end{equation}
By Claim~\ref{b-bound} and $b\geq 3$, we obtain $F < \phi-\pi\sqrt{3}/2$, so the left-hand side of \eqref{eq:quadratic} is decreasing in $F$.
Thus we can substitute $F=2$ into \eqref{eq:quadratic} to obtain
\begin{equation*}
    4-2(2\phi-\pi\sqrt{3})+(1-\pi\sqrt{3}/6)\phi^2>0,
\end{equation*}
hence $\phi < 4.114\dots$ or $\phi>38.849\dots$, or in terms of $n$, $n < 4.468\dots$ or $n>146.199\dots$.
Since $n\geq 5$ by assumption, we conclude that $n\geq147$, as required.
For the second part of the claim, we solve for $F$ in \eqref{eq:quadratic} to obtain
\begin{equation*}
    F < \phi-\pi\sqrt{3}/2-\frac12\sqrt{\frac{2\pi}{\sqrt3}\phi^2-4\pi\sqrt3\phi+3\pi^2}.
\end{equation*}
In order to show that $F < (\phi-8)/11$, it is sufficient to show that
\begin{equation*}
    \phi-\pi\sqrt{3}/2-\frac12\sqrt{\frac{2\pi}{\sqrt3}\phi^2-4\pi\sqrt3\phi+3\pi^2} \leq \frac{\phi-8}{11}.
    \end{equation*}
This inequality is equivalent to $\phi \leq 2.084\dots$ or $\phi \geq 20.506\dots$.
However, we have already shown that $\phi>38.849\dots$.
Thus we conclude that $F < (\phi-8)/11$, which proves the second part of the claim.
\end{proof}

\subsection{Lattice components}
We define a \emph{lattice component} of $G$ to be any maximal $2$-connected subgraph (on at least $3$ vertices) that lies on some triangular lattice.
Denote the lattice components of $G$ by $G_1,\dots,G_k$.
Denote the number of vertices by $n_i=n(G_i)$, the number of edges by $e_i=e(G_i)$, and the boundary length of $G_i$ by $b_i$.
Assume that $n_1\geq n_2\geq\dots\geq n_k$.
Note that no two lattice components have an edge in common, otherwise their union would be a larger $2$-connected graph on the same lattice.

\begin{claim}\label{lattice-boundary}
For each $i=1,\dots,k$, $b_i\geq \sqrt{12n_i-3}-3$.
\end{claim}
\begin{proof}
$G_i$ is on a triangular lattice and is $2$-connected, so its boundary is a cycle.
Construct a new graph $G'$ from $G_i$, by filling up all missing lattice vertices and edges inside the boundary of $G_i$.
By Theorem~\ref{thm0}, $e(G')\leq 3n(G')-\sqrt{12n(G')-3}$, since $G'$ is also on the triangular lattice.
(We cannot use induction here, because $G'$ might have more than $n$ vertices.)
By \eqref{basic-identity} applied to $G'$ (which is still $2$-connected and with $F=0$), we obtain $b_i\geq\sqrt{12n(G')-3}-3\geq\sqrt{12n_i-3}-3$, since $G'$ has the same outer boundary as $G_i$.
\end{proof}

We will later need that the largest lattice component is not too small.
We first show that the lattice components cover almost all of $G$ and do not overlap too much.
\begin{claim}\label{lattice-components}
$n-2F\leq \sum_{i=1}^k n_i \leq n+4F$.
\end{claim}
\begin{proof}
For the lower bound, consider a vertex $v$ of $G$ that does not belong to any lattice component.
Let $s$ be the number of inner faces of $G$ incident to $v$.
Since $v$ has at least $3$ neighbours (Claim~\ref{mindegree}), $s\geq 2$.
Assign a charge of $1/s$ to each inner face incident to $v$.
Since these faces are all non-triangular, the total charge is $\leq \frac12\sum_{i\geq 4} if_i\leq 2F$.
The total charge on all such vertices counts the number of vertices not covered by any lattice component.
Thus, at least $n-2F$ vertices of $G$ are covered by the lattice components, which gives $\sum_{i=1}^k n_i \geq n-2F$.

For the upper bound, consider a vertex $v$ that belongs to $t\geq 2$ of the $G_i$.
Now let $s$ be the number of non-triangular inner faces of $G$ incident to $v$.
Since two adjacent faces incident to $v$ cannot belong to distinct $G_i$, there is at least one non-triangular face between any two faces at $v$ belonging to distinct $G_i$, so $v$ is incident to at least $t$ non-triangular faces, of which one could be the outer face.
Hence $s\geq t-1$.
Assign a charge of $(t-1)/s$ to each non-triangular inner face incident to $v$.
This gives a total charge of $t-1$ for each $v$, which is exactly its contribution to $\sum_i n_i - n$.
Since $(t-1)/s\leq 1$, we obtain that the total charge is at most $\sum_{i\geq 4}if_i\leq 4F$.
Thus, $\sum_i n_i \leq n+4F$.
\end{proof}

From now on we concentrate on the largest lattice component $G_1$.
We first show that it covers a considerable proportion of $G$.

\begin{claim}\label{n1bound}
$n_1 > 3n/4$.
\end{claim}
\begin{proof}
Suppose that all $n_i\leq 3n/4$.
We will find a contradiction by bounding $\sum_{i=1}^k b_i$ from above and below.
Since no two $G_i$ have a common boundary edge, this sum counts the number of edges on the boundaries of all lattice components.

For the upper bound, note that if an edge is on the boundary of some lattice component then it borders a non-triangular face on the outside of the lattice component. This gives
\begin{equation}\label{b_i-upper}
\sum_{i=1}^k b_i \leq b + \sum_{i\geq 4} i f_i < \sqrt{12n-3} -3 -F + 4F < \frac{14}{11}\sqrt{12n-3}-6
\end{equation}
from Claims~\ref{b-bound} and \ref{Fbound1}.
For the lower bound we just use Claim~\ref{lattice-boundary}:
\begin{equation}\label{b_i-lower}
\sum_{i=1}^k b_i \geq \sum_{i=1}^k (\sqrt{12n_i-3}-3) =\sum_{i=1}^k\phi(n_i).
\end{equation}
We lower bound the right-hand side by relaxing it to a continuous optimisation problem.
We have $k\leq N/3$, where $N=\sum_{i=1}^k n_i$.
For later reference, note that $3n/4+3 < N < 6n/4 - 3$ since $n\geq 13$, by Claims~\ref{lattice-components} and \ref{Fbound1}.
In particular, $k\geq 2$.
For each $\ell=2,\dots,\lfloor N/3\rfloor$, define $\Phi_\ell(x_1,\dots,x_\ell)=\sum_{i=1}^\ell \phi(x_i)$ on the domain
\[ D_\ell = \left\{(x_1,\dots,x_\ell)\in\bR^\ell : \tfrac34n \geq x_1\geq\dots \geq x_\ell\geq 3, \sum_{i=1}^\ell x_i = N\right\}.\]
For each $\ell$, the function $\Phi_\ell$ has a minimum value $y_\ell$ on $D_\ell$.
Among all of these $\ell$, fix one, say $m$, that minimises $y_\ell$.
Fix $(x_1,\dots,x_m)\in D_m$ such that $\Phi_m(x_1,\dots,x_m)=y_m$.
By strict concavity of $\phi$ (Lemma~\ref{concave}) we cannot have two $x_i$ in the open interval $(3,\frac34n)$.
If $m\geq 3$, we cannot have some $x_i=3$ and another $x_j\leq 3n/4-3$, since we can then replace both by a single variable equal to $x_j+3$, thereby finding \begin{align*}
y_{m-1} &\leq y_m +\phi(3+x_j)-\phi(x_j)-\phi(3)\\
&\leq y_m +\phi(6)-\phi(3)-\phi(3)\quad\text{by Lemma~\ref{concave}}\\
&< y_m.
\end{align*}
It follows that we cannot have any $x_i=3$, since this would imply that all other $x_j\in(3n/4-3,3n/4]$, but $3+3n/4 < N$ and $3+2(3n/4-3) > N$ as noted before.
We cannot have two $x_i$ equal to $\frac34n$, since $6n/4 > N$.
Thus necessarily $m=2, x_1=\frac34n, x_2=N-\frac34n$.
It follows that
\begin{align}
    \sum_{i=1}^k\phi(n_i) &= \Phi_k(n_1,\dots,n_k)\geq \Phi_2(x_1,x_2)\notag \\
    &\geq\Phi_2(x_1+1/16,x_2-1/16) \quad\text{by Lemma~\ref{concave}}\notag\\
    &= \sqrt{(3/4)(12n-3)} - 3 + \sqrt{12(N-3n/4)-15/4}-3\notag\\
    &\geq \frac{\sqrt{3}}{2}\sqrt{12n-3} + \sqrt{12(n/4-2F)-15/4} - 6\quad\text{by Claim~\ref{lattice-components}.}\label{phi-lower}
\end{align}

Putting \eqref{b_i-upper}, \eqref{b_i-lower} and \eqref{phi-lower} together, we obtain
\[ \frac16(12n-3) > \Bigl(\frac{14}{11}-\frac{\sqrt{3}}{2}\Bigr)^2(12n-3) > 12(n/4-2F)-15/4 = \frac14(12n-3)-3-24F.\]
Again using Claim~\ref{Fbound1} it follows that
\[ \frac{1}{12}(12n-3)-3 < 24F < \frac{24}{11}\sqrt{12n-3}-24,\]
which has no solutions in $n$, a contradiction.
\end{proof}

The next inequality comes from applying the induction hypothesis to $G_1$ and a slight enlargement of the remainder $G-G_1$.
\begin{claim}\label{eq1}
$\sqrt{12(n-n_1)} < 6F + \sqrt{12n-3}-\sqrt{12n_1-3}$.
\end{claim}
\begin{proof}
Let $K$ be the set of vertices of $G_1$ that are joined to some vertex in $G-G_1$, and let $G'$ be the subgraph of $G$ induced by $V(G-G_1)\cup K$.
Since $G_1$ is on the lattice, we already know by Harborth's Theorem~\ref{thm0} that $e_1\leq 3n_1-\sqrt{12n_1-3}$.
If $K=\emptyset$, then $G_1=G$ since $G$ and $G_1$ are connected, and then the assumption \eqref{assumption1} is contradicted.
Thus $K\neq\emptyset$.
Let $n'=n(G') = n-n_1+|K|$ and $e'=e(G')$.
Then $G'$ and $G_1$ cover all the edges of $G$, hence $e\leq e_1+e'$.
To bound $e'$ from above, we will apply induction, but then we need to ensure that $K\neq V(G_1)$.
We do this by bounding $|K|$ from above.

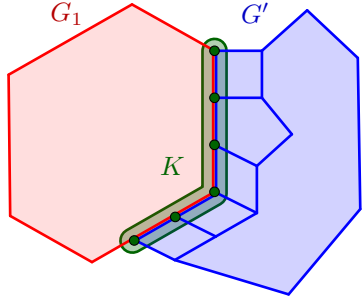
\begin{figure}
    \centering
\definecolor{ffqqqq}{rgb}{1,0,0}
\definecolor{qqqqct}{rgb}{0,0,0.7647058823529411}
\definecolor{bvqqqq}{rgb}{0.7098039215686275,0,0}
\definecolor{qqqqff}{rgb}{0,0,1}
\definecolor{uuuuuu}{rgb}{0.26666666666666666,0.26666666666666666,0.26666666666666666}
\definecolor{qqwuqq}{rgb}{0,0.39215686274509803,0}
\begin{tikzpicture}[line cap=round,line join=round,>=triangle 45,scale=1.2, rotate=90]
\draw[color=qqwuqq, line width=10pt] (0.8660254037844386,0) -- (-0.71282032302755,0) -- (-1.2328970236140115,0.8997334191343);
\draw[color=qqwuqq!30!white, line width=8pt] (0.8660254037844386,0) -- (-0.71282032302755,0) -- (-1.2328970236140115,0.8997334191343);
\fill[line width=2pt,color=qqqqff,fill=qqqqff,fill opacity=0.18] (-1.2128970236140115,0.8997334191343) -- (-1.4464476158121065,0.4355630615241674) -- (-1.8256607784817434,-0.8118273618016772) -- (-0.8830371369361856,-1.5973470630896291) -- (0.7870743772396426,-1.5701497327271368) -- (1.3131470128738905,-1.0152047894564749) -- (0.8660254037844385,-0.519615242270663) -- (0.8660254037844386,0) -- (0.34641016151377557,0) -- (-0.17320508075688718,0) -- (-0.69282032302755,0) -- (-0.9528586733207809,0.44986670956715014) -- cycle;
\fill[line width=2pt,color=ffqqqq,fill=ffqqqq,fill opacity=0.13] (-0.9585597254676749,2.239470744218566) -- (-1.4691394169768557,1.3430331179047497) -- (-0.69282032302755,0) -- (0.8660254037844386,0) -- (1.378177858738979,0.9042675836742334) -- (0.6013407899841223,2.2570884027808744) -- cycle;
\draw [line width=1pt,color=qqqqff] (-0.17320508075688718,0)-- (-0.40675567295498205,-0.4641703576101331);
\draw [line width=1pt,color=qqqqff] (-0.40675567295498205,-0.4641703576101331)-- (-0.05637934900207353,-0.8478843547658871);
\draw [line width=1pt,color=qqqqff] (-0.05637934900207353,-0.8478843547658871)-- (0.3464101615137755,-0.5196152422706628);
\draw [line width=1pt,color=qqqqff] (-0.69282032302755,0)-- (-0.9263709152256449,-0.464170357610133);
\draw [line width=1pt,color=qqqqff] (-0.9263709152256449,-0.464170357610133)-- (-0.40675567295498205,-0.4641703576101331);
\draw [line width=1pt,color=qqqqff] (-0.9263709152256449,-0.464170357610133)-- (-1.1864092655188756,-0.014303648042982953);
\draw [line width=1pt,color=qqqqff] (-1.1864092655188756,-0.014303648042982953)-- (-0.9528586733207809,0.44986670956715014);
\draw [line width=1pt,color=qqqqff] (-1.4464476158121065,0.4355630615241673)-- (-1.1864092655188756,-0.014303648042982953);
\draw [line width=1pt,color=qqqqff] (0.34641016151377557,0)-- (0.3464101615137755,-0.5196152422706628);
\draw [line width=1pt,color=qqqqff] (0.3464101615137755,-0.5196152422706628)-- (0.8660254037844385,-0.519615242270663);
\draw [color=bvqqqq](1.5,1.9) node[anchor=north west] {$G_1$};
\draw [color=qqqqct](1.5,-0.19) node[anchor=north west] {$G'$};
\draw [line width=1pt,color=qqqqff] (-1.2128970236140115,0.8997334191343)-- (-1.4464476158121065,0.4355630615241674);
\draw [line width=1pt,color=qqqqff] (-1.4464476158121065,0.4355630615241674)-- (-1.8256607784817434,-0.8118273618016772);
\draw [line width=1pt,color=qqqqff] (-1.8256607784817434,-0.8118273618016772)-- (-0.8830371369361856,-1.5973470630896291);
\draw [line width=1pt,color=qqqqff] (-0.8830371369361856,-1.5973470630896291)-- (0.7870743772396426,-1.5701497327271368);
\draw [line width=1pt,color=qqqqff] (0.7870743772396426,-1.5701497327271368)-- (1.3131470128738905,-1.0152047894564749);
\draw [line width=1pt,color=qqqqff] (1.3131470128738905,-1.0152047894564749)-- (0.8660254037844385,-0.519615242270663);
\draw [line width=1pt,color=qqqqff] (0.8660254037844385,-0.519615242270663)-- (0.8660254037844386,0);
\draw [line width=1pt,color=qqqqff] (0.8660254037844386,-0.017)-- (0.34641016151377557,-0.017);
\draw [line width=1pt,color=qqqqff] (0.34641016151377557,-0.017)-- (-0.17320508075688718,-0.017);
\draw [line width=1pt,color=qqqqff] (-0.17320508075688718,-0.017)-- (-0.69282032302755,-0.017);
\draw [line width=1pt,color=qqqqff] (-0.71282032302755,-0.03)-- (-0.9728586733207809,0.41986670956715014);
\draw [line width=1pt,color=qqqqff] (-0.9728586733207809,0.41986670956715014)-- (-1.2328970236140115,0.8697334191343);
\draw [line width=1pt,color=ffqqqq] (-0.9585597254676749,2.239470744218566)-- (-1.4691394169768557,1.3430331179047497);
\draw [line width=1pt,color=ffqqqq] (-1.4691394169768557,1.3430331179047497)-- (-0.69282032302755,0);
\draw [line width=1pt,color=ffqqqq] (-0.69282032302755,0.015)-- (0.8660254037844386,0.015);
\draw [line width=1pt,color=ffqqqq] (0.8660254037844386,0)-- (1.378177858738979,0.9042675836742334);
\draw [line width=1pt,color=ffqqqq] (1.378177858738979,0.9042675836742334)-- (0.6013407899841223,2.2570884027808744);
\draw [line width=1pt,color=ffqqqq] (0.6013407899841223,2.2570884027808744)-- (-0.9585597254676749,2.239470744218566);
\draw [color=qqwuqq](-0.2,0.7) node[anchor=north west] {$K$};
\draw [fill=qqwuqq] (0.8660254037844386,0) circle (1.5pt); 
\draw [fill=qqwuqq] (0.34641016151377557,0) circle (1.5pt); 
\draw [fill=qqwuqq] (-0.17320508075688718,0) circle (1.5pt); 
\draw [fill=qqwuqq] (-0.69282032302755,0) circle (1.5pt); 
\draw [fill=qqwuqq] (-0.97,0.43) circle (1.5pt);
\draw [fill=qqwuqq] (-1.23,0.88) circle (1.5pt);
\end{tikzpicture}
    \caption{Decomposition of $G$ into $G_1$ and $G'$ with common vertex set $K$}\label{eq1fig}
\end{figure}    
Each vertex $v$ in $K$ belongs to some inner face of $G$ not lying on the lattice of $G_1$, otherwise all inner faces around $v$ lie on the same lattice, so must be part of $G_1$ by maximality.
Consider any inner face $\Gamma$ not lying on the lattice of $G_1$, so with at least one vertex not in $K$.
Denote the length of the cycle bounding $\Gamma$ by $i$.
Suppose that all but one of the vertices on this cycle lie in $K$ and consider the remaining vertex $v$ with neighbours $u$ and $w$ on $\Gamma$.
Since $u$ and $w$ are on the lattice of $G_1$, $v$ also lies on the same lattice by Lemma~\ref{latticelemma}, but then $\Gamma$ lies on the lattice of $G_1$, a contradiction.
Hence at most $i-2$ vertices of $K$ belong to $\Gamma$.
Since $\Gamma$ is a non-triangular face of $G$, we find the upper bound
\begin{equation*}
    |K| \leq \sum_{i\geq4}(i-2)f_i \leq 2F.
\end{equation*}
By Claims~\ref{Fbound1} and \ref{n1bound} it follows that $2F < n_1$, hence
$K$ cannot be all of $G_1$ so $G'$ is a proper subset of $G$, and we can apply induction to $G'$ to obtain \[e\leq e_1+e' \leq 3n_1-\sqrt{12n_1-3} + 3n'-\sqrt{12n'-3}.\]
Using assumption~\eqref{assumption1} and $n_1+n'=n+|K|$, we obtain
\begin{equation}\label{eq2}\sqrt{12(n-n_1+|K|)-3} < 3|K| + \sqrt{12n-3}-\sqrt{12n_1-3}.
\end{equation}
The claim now follows from the bounds $1\leq |K|\leq 2F$.
\end{proof}
Let $b_*$ be the number of boundary edges of $G$ not on the boundary of $G_1$.
\begin{claim}\label{b*upperbound}
$b_* < \sqrt{12n-3}-\sqrt{12n_1-3}$.
\end{claim}
\begin{proof}
Each edge on the boundary of $G_1$ borders either the outer face of $G$ or an inner non-triangular face $\Gamma$ of $G$.
Let $i$ denote the length of the cycle bounding $\Gamma$.
Not all the edges of $\Gamma$ are on the boundary of $G_1$.
Thus $\Gamma$ contributes at most $i-1$ edges to the boundary of $G_1$.
If $\Gamma$ has $i-1$ or $i-2$ of its edges on the boundary of $G_1$, then by Lemma~\ref{latticelemma}, $\Gamma$ has to be a lattice polygon, so has to be part of $G_1$ by maximality.
Thus $\Gamma$ has at most $i-3$ of its edges on the boundary of $G_1$.
Therefore, \[b_1\leq b-b_* + \sum_{i\geq 4} (i-3)f_i = b-b_* + F.\]
The claim now follows from the inequalities in Claims~\ref{b-bound} and \ref{lattice-boundary}.
\end{proof}

\subsection{An improved isoperimetric inequality}
Our next aim is to find a better upper bound for $F$ in a form similar to the upper bound of $b_*$ in Claim~\ref{b*upperbound}.
This bound (Claim~\ref{Fbound2}) will then be combined with Claim~\ref{eq1} to obtain a contradiction.
Claim~\ref{Fbound2} will follow from our isoperimetric inequality in Lemma~\ref{isoperimetric}, which we now restate and prove.
\begin{isoperimetriclemma}
Let $P$ be a simple polygon with perimeter $b$, area $A$, and with the total length of the sides not parallel to any side of some fixed regular hexagon at most $b_*$.
Then \[8\sqrt{3}A \leq \bigl(b+ \bigl(\tfrac{2}{\sqrt{3}}-1\bigr)b_*\bigr)^2.\]
\end{isoperimetriclemma}
\begin{proof}
Let $h$ be a fixed regular hexagon.
For any simple polygonal path (open or closed) $Q$, we denote its length by $b(Q)$ and the total length of its sides not parallel to any side of $h$ by $b_*(Q)$. If $Q$ is closed (a polygon), we denote its area by $A(Q)$.

We first modify $P$ into a convex polygon $P'$ with $b(P)=b(P')$, $b_*(P)=b_*(P')$, and $A(P)\leq A(P')$.
We say that a simple polygon $Q$ is a \emph{rearrangement} of $P$ if there is a bijection between the edge set of $Q$ and the edge set of $P$ such that corresponding edges are parallel and have the same length.
Note that for any rearrangement $Q$ of $P$ we have $b(P)=b(Q)$ and $b_*(P)=b_*(Q)$.
Among all rearrangements of $P$ fix one, $P'$, of maximum area.
We claim that $P'$ is convex.
If $P'$ is not convex, then it has two vertices $p$ and $q$ such that the boundary of $P'$ between $p$ and $q$ lies in the interior of the convex hull of $P'$.
If we replace this part of the boundary by its rotation by $180^\circ$ around the midpoint of $p$ and $q$, then we obtain a new simple polygon that is still a rearrangement of $P$, but with larger area.
\begin{figure}
    \centering
    \definecolor{wwwwww}{rgb}{0.5,0.5,0.8}
\begin{tikzpicture}[line cap=round,line join=round,>=triangle 45,x=1cm,y=1cm,scale=2.5]
\draw [line width=0.8pt,color=wwwwww] (-1.299038105676658,0.75)-- (-2.577738284548175,-0.009039078580062837);
\draw [line width=0.8pt,color=wwwwww] (-1.299038105676658,-1.75)-- (0.4330127018922193,-0.75);
\draw [line width=0.8pt,color=wwwwww] (0.4330127018922193,-0.75)-- (0.4330127018922193,-0.25);
\draw [line width=0.8pt,color=wwwwww] (0.4330127018922193,-0.25)-- (-1.299038105676658,0.75);
\draw [line width=0.8pt,color=wwwwww] (-2.5784539865037384,-0.009463920745462581)-- (-1.2988669845147736,0.7501015778766504);
\draw [line width=0.8pt,color=wwwwww] (-2.5784539865037384,-0.009463920745462581)-- (-2.577738284548175,-1.0090390785800625);
\draw [line width=0.8pt,color=wwwwww] (-2.5783385664466625,-1.0086912367356922)-- (-1.2982599871104685,-1.7504508918191333);
\draw [line width=0.8pt,color=wwwwww] (-1.2984407519793464,-1.7496551176820556)-- (0.43152351708770187,-0.7508597812477612);
\draw [line width=0.8pt,color=wwwwww] (0.4330127018922193,-0.7522891501191088)-- (0.4330127018922193,-0.2504734553841228);
\draw [line width=0.8pt,color=wwwwww] (-1.2993386140245669,0.7501734985755589)-- (0.43156815547079697,-0.2491659907347349);
\draw [line width=1.2pt] (-0.9781586665683353,-1.5647401694533936)-- (0.3015124589344671,-0.8259217006701594);
\draw [line width=1.2pt] (-0.98938751086056,0.571223145728198)-- (0.2480774300967321,-0.14322757104621875);
\draw [color=wwwwww](-0.3,0.33) node[anchor=north west] {$H$};
\draw (-2.792560451488531,0.13) node[anchor=north west] {$r_1$};
\draw (-2.8003408527730677,-0.9223242063932245) node[anchor=north west] {$r_2$};
\draw (-1.3375790661423126,-1.75) node[anchor=north west] {$r_3$};
\draw (0.42,-0.73) node[anchor=north west] {$r_4$};
\draw (0.4380725587374475,-0.1) node[anchor=north west] {$r_5$};
\draw (-1.3758680499672131,0.95) node[anchor=north west] {$r_6$};
\draw (0.17483579494125667,0.1) node[anchor=north west] {$p_6$};
\draw (-2.8308494353134317,-0.1230416690484299) node[anchor=north west] {$p_2$};
\draw (-1.7300411503475426,-1.6) node[anchor=north west] {$q_3$};
\draw (-1,-1.55) node[anchor=north west] {$p_4$};
\draw (0.22,-0.85) node[anchor=north west] {$q_4$};
\draw (0.47,-0.52) node[anchor=north west] {$p_5$};
\draw (-2.390526121327076,0.3842873666314996) node[anchor=north west] {$q_1$};
\draw (-0.95,0.74) node[anchor=north west] {$q_6$};
\draw (-2.15,0.7384604670118278) node[anchor=north west] {$p_7=p_1$};
\draw (-1.3040762052955246,-0.34320332604160686) node[anchor=north west] {$P'$};
\draw [shift={(-0.04249048800425103,-0.5759370311096139)},line width=1.2pt]  plot[domain=0.42390390675380274:1.0100513713385812,variable=\t]({1*0.5216767307610695*cos(\t r)+0*0.5216767307610695*sin(\t r)},{0*0.5216767307610695*cos(\t r)+1*0.5216767307610695*sin(\t r)});
\draw [shift={(0.17409511750684015,-0.5976077115173691)},line width=1.2pt]  plot[domain=5.203950402602928:6.226137888453543,variable=\t]({1*0.26156860984414165*cos(\t r)+0*0.26156860984414165*sin(\t r)},{0*0.26156860984414165*cos(\t r)+1*0.26156860984414165*sin(\t r)});
\draw [line width=1.2pt] (0.4330127018922193,-0.3613599404968493)-- (0.4330127018922193,-0.6123943354195447);
\draw [line width=1.2pt] (-2.4400431862482064,-1.0888284526625736)-- (-1.6066261233711758,-1.5717637685717705);
\draw [line width=1.2pt] (-2.577898898975092,-0.784719178532384)-- (-2.578312866726034,-0.20655689076723055);
\draw [shift={(-2.05515323250691,-0.3214796474927902)},line width=1.2pt]  plot[domain=2.067369558065851:2.9253566254116445,variable=\t]({1*0.5358318629319668*cos(\t r)+0*0.5358318629319668*sin(\t r)},{0*0.5358318629319668*cos(\t r)+1*0.5358318629319668*sin(\t r)});
\draw [shift={(-2.1720132949418964,-0.7872217430124655)},line width=1.2pt]  plot[domain=3.135427042531723:3.9858669713759087,variable=\t]({1*0.4058933189772516*cos(\t r)+0*0.4058933189772516*sin(\t r)},{0*0.4058933189772516*cos(\t r)+1*0.4058933189772516*sin(\t r)});
\draw (-2.7973465744666437,-0.68) node[anchor=north west] {$q_2$};
\draw (-2.6,-1.1) node[anchor=north west] {$p_3$};
\draw (0.47,-0.27) node[anchor=north west] {$q_5$};
\draw (-1.4015096810328753,-1.3500001676203403) node[anchor=north west] {$Q_4$};
\draw (0.10440804089748178,-0.22833637456690584) node[anchor=north west] {$Q_6$};
\draw (-1.4027929431642,0.6331657614933518) node[anchor=north west] {$Q_1$};
\draw (-2.502745720042539,0.010969774338721289) node[anchor=north west] {$Q_2$};
\draw (-2.481462457911215,-0.7548099021592856) node[anchor=north west] {$Q_3$};
\draw (0.13176068813824363,-0.55) node[anchor=north west] {$Q_5$};
\draw [line width=1.2pt] (-1.5950270580690447,0.5743003438913497)-- (-2.310431814428669,0.14963459655546396);
\begin{scriptsize}
\draw [fill=black] (-0.9724057394937612,0.5614187154582602) circle (0.6pt);
\draw [fill=black] (-1.5950270580690447,0.5743003438913497) circle (0.6pt);
\draw [fill=black] (-2.310431814428669,0.14963459655546396) circle (0.6pt);
\draw [fill=black] (-0.9528625341176387,-1.5501354405735277) circle (0.6pt);
\draw [fill=black] (0.29755628816173996,-0.8282057969307535) circle (0.6pt);
\draw [fill=black] (0.23425672761327815,-0.135248184747007) circle (0.6pt);
\draw [fill=black] (-1.6066261233711758,-1.5717637685717705) circle (0.6pt);
\draw [fill=black] (0.4330127018922193,-0.3613599404968493) circle (0.6pt);
\draw [fill=black] (0.4330127018922193,-0.6123943354195447) circle (0.6pt);
\draw [fill=black] (-2.578312866726034,-0.20655689076723055) circle (0.6pt);
\draw [fill=black] (-2.577898898975092,-0.784719178532384) circle (0.6pt);
\draw [fill=black] (-2.4400431862482064,-1.0888284526625736) circle (0.6pt);
\draw [shift={(-1.294964567911478,0.022339296748768746)},line width=1.2pt]  plot[domain=1.0315900459052108:2.0687356603703915,variable=\t]({1*0.6282123983704662*cos(\t r)+0*0.6282123983704662*sin(\t r)},{0*0.6282123983704662*cos(\t r)+1*0.6282123983704662*sin(\t r)});
\draw [shift={(-1.2991778217593082,-0.9648341525682973)},line width=1.2pt]  plot[domain=4.130907994781042:5.246673714994852,variable=\t]({1*0.6803587413608682*cos(\t r)+0*0.6803587413608682*sin(\t r)},{0*0.6803587413608682*cos(\t r)+1*0.6803587413608682*sin(\t r)});
\end{scriptsize}
\end{tikzpicture}
    \caption{The polygon $P'$ and its circumscribed hexagon $H$}
    \label{isolemfig}
\end{figure}
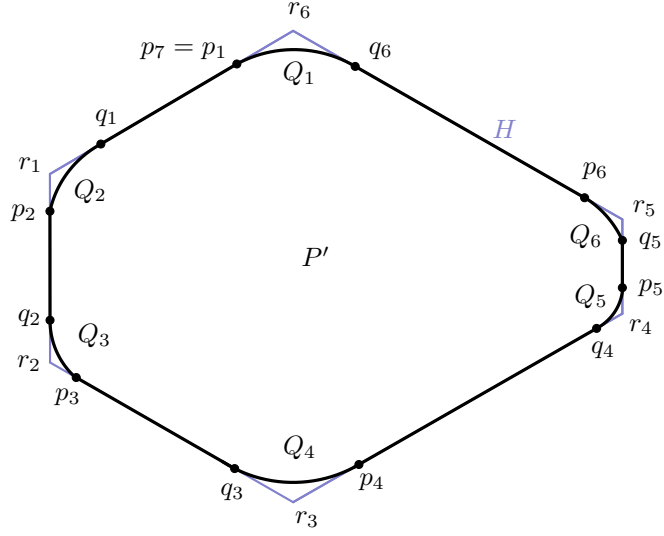
Let $H$ denote the hexagon with sides parallel to the sides of $h$ circumscribed around $P'$.
Number the sides of $H$ in order from $1$ to $6$.
For each $i=1,\dots,6$, let $p_i q_i$ be the segment of $P'$ on side $i$ of $H$, such that the points lie in the order $p_1,q_1,p_2,q_2,\dots,p_6,q_6$ around $P'$.
Denote the path of $P'$ between $q_i$ and $p_{i+1}$ by $Q_i$, $i=1,\dots,6$, where $p_7=p_1$.

Note that in a triangle $pqr$ with $\angle r=120^\circ$, the cosine rule gives
\begin{align*}
    |pq|^2 &= |pr|^2+|qr|^2-2|pr|\cdot|qr|\cos120^\circ\\
    &= |pr|^2+|qr|^2+|pr|\cdot|qr|\\
    &= \tfrac34(|pr|+|qr|)^2 + \tfrac14(|pr|-|qr|)^2\\
    &\geq \tfrac34(|pr|+|qr|)^2,
\end{align*}
hence $|pr|+|qr|\leq\frac{2}{\sqrt{3}}|pq|$.
We now apply this to each triangle $p_{i+1}q_ir_i$ to obtain an upper bound on the perimeter of $H$ in terms of $b$ and $b_*$:
\begin{align*}
    b &= \sum_{i=1}^6 |p_iq_i| +\sum_{i=1}^6 b(Q_i)\\
    &= b(H) -\sum_{i=1}^6|q_ir_i|-\sum_{i=1}^6|r_ip_{i+1}|+\sum_{i=1}^6 b(Q_i)\\
    &\geq b(H) -\sum_{i=1}^6|q_ir_i|-\sum_{i=1}^6|r_ip_{i+1}|+\sum_{i=1}^6|q_ip_{i+1}| \quad\text{by the triangle inequality}\\
    &\geq b(H) -\tfrac{2}{\sqrt{3}}\sum_{i=1}^6|q_ip_{i+1}|+\sum_{i=1}^6|q_ip_{i+1}| \quad\text{since $\angle r_i=120^\circ$}\\
    &= b(H) -(\tfrac{2}{\sqrt{3}}-1)\sum_{i=1}^6|q_ip_{i+1}|\\
    &\geq b(H) -(\tfrac{2}{\sqrt{3}}-1)\sum_{i=1}^6 b(Q_i)\quad\text{again by the triangle inequality}\\
    &= b(H) -(\tfrac{2}{\sqrt{3}}-1)b_*.
\end{align*}
By the isoperimetric inequality for hexagons (or L'Huilier's inequality for hexagons with sides parallel to those of a fixed hexagon) the hexagon with a fixed perimeter maximising the area is regular.
It follows that
\begin{align*}
A&\leq A(P')\leq A(H)\\
&\leq \tfrac{\sqrt{3}}{24}b(H)^2\quad\text{by the isoperimetric inequality for hexagons}\\
&\leq\tfrac{\sqrt{3}}{24}(b+(\tfrac{2}{\sqrt{3}}-1)b_*)^2,
\end{align*}
and Lemma~\ref{isoperimetric} follows.
\end{proof}
We now obtain an improved upper bound for $F$ from the above isoperimetric estimate, again bounding the area from below by using the lower bound on $f_3$ from Claim~\ref{f3bound}.
\begin{claim}\label{Fbound2}
$F < \frac16(\sqrt{12n-3}-\sqrt{12n_1-3}) + \frac12$.
\end{claim}
\begin{proof}
As before, the area $A$ of the boundary polygon of $G$ is at least $\frac{\sqrt{3}}{4}f_3$.
If we substitute this, as well as the upper bounds for $b$ and $b_*$ from Claims~\ref{b-bound} and \ref{b*upperbound}, respectively, and the lower bound for $f_3$ from Claim~\ref{f3bound}, into Lemma~\ref{isoperimetric}, we obtain
\[12n+6-6\sqrt{12n-3}-6F \leq (b+cD)^2 < (\sqrt{12n-3}-3+c D - F)^2,\] where $D=\sqrt{12n-3}-\sqrt{12n_1-3}$ and $c=2/\sqrt{3}-1$.

Suppose that $c D-F \leq -1/2$.
Then \[12n+6-6\sqrt{12n-3}-6F < (\sqrt{12n-3} - 7/2)^2.\]
Multiplying out and simplifying, we obtain $F > \frac16\sqrt{12n-3} - \frac{13}{24}$, which contradicts the previous upper bound on $F$ from Claim~\ref{Fbound1}.

Therefore, $c D-F > -1/2$, which gives the claim.
\end{proof}

\subsection{The conclusion}
We now combine the new estimate for $F$ from Claim~\ref{Fbound2} with Claim~\ref{eq1} to obtain a contradiction.
Claim~\ref{Fbound2} together with $F\geq 2$ (Claim~\ref{Fgeq2}) imply that
\begin{equation}\label{*}
D:=\sqrt{12n-3}-\sqrt{12n_1-3} > 9.
\end{equation}
Again by Claim~\ref{Fbound2},
\[6F < D+3 < D+\frac13D = \frac43D\quad\text{by \eqref{*}}. \]
Then Claim~\ref{eq1} gives
\[ \sqrt{12(n-n_1)} < \frac43D + D = \frac73D = \frac73\cdot\frac{12(n-n_1)}{\sqrt{12n-3}+\sqrt{12n_1-3}}.\]
It follows that
\begin{equation}\label{**}
\frac37(\sqrt{12n-3}+\sqrt{12n_1-3}) < \sqrt{12(n-n_1)}.
\end{equation}
By Claim~\ref{n1bound}, and using that $n_1$ and $n$ are integers, we have $n_1 > \frac34n +\frac{1}{16}$, which implies $\sqrt{12n_1-3} > \frac{\sqrt{3}}{2}\sqrt{12n-3}$ and $\sqrt{12(n-n_1)} < \frac12\sqrt{12n-3}$.
Together with \eqref{**} we obtain the contradiction $\frac37(1+\frac{\sqrt{3}}{2}) < \frac12$.

This finishes the induction step and thereby the proof of Theorem~\ref{thm1}.
\qed

\end{document}